\documentclass[10pt]{article}
\usepackage{amsmath,amsthm,amsfonts,amssymb,amscd, amsxtra}
\usepackage[margin=2.5cm,nohead]{geometry}
\theoremstyle{plain}
\newtheorem{theorem}{Theorem}[section]
\newtheorem{lemma}{Lemma}[section]
\newtheorem{definition}{Definition}[section]

\newtheorem{proposition}{Proposition}[section]

\newtheorem{remark}{Remark}[section]
\newtheorem{example}{Example}[section]

\newcommand{\beq}{\begin{equation}}
\newcommand{\eeq}{\end{equation}}
\newcommand{\beqa}{\begin{eqnarray}}
\newcommand{\eeqa}{\end{eqnarray}}
\newcommand{\beqas}{\begin{eqnarray*}}
\newcommand{\eeqas}{\end{eqnarray*}}
\newenvironment{ProofMainTheo}{{\noindent\bf Proof of Theorem 3.1.}}{\hfill$\Box$\\}

\def\min{\operatorname{min}}
\def\max{\operatorname{max}}

\makeatletter
\renewcommand*{\@biblabel}[1]{\hfill#1.}
\makeatother

\begin{document}

\title{Proximal point method for a special class of nonconvex multiobjective optimization functions}

\author{
G. C. Bento\thanks{The author was partially supported by CAPES-MES-CUBA 226/2012, FAPEG 201210267000909 - 05/2012 and CNPq Grants 458479/2014-4, 471815/2012-8, 303732/2011-3, 236938/2012-6, 312077/2014-9. . IME-Universidade Federal de Goi\'as,
Goi\^ania-GO 74001-970, BR (Email: {\tt glaydston@ufg.br}) - {\bf Corresponding author}}
\and
O. P. Ferreira
\thanks{The author was partially supported by FAPEG 201210267000909 - 05/2012, PRONEX--Optimization(FAPERJ/CNPq), CNPq Grants 4471815/2012-8,  305158/2014-7. IME-Universidade Federal de Goi\'as,
Goi\^ania, GO 74001-970, BR (Email: {\tt orizon@mat.ufg.br}).}\and
V. L. Sousa Junior  \thanks{This author was partially supported by CAPES and CNPq.  IME-Universidade Federal de Goi\'as,
Goi\^ania, GO 74001-970, BR  (Email: {\tt valdinesldjs@gmail.com}).}
}
\date{}
\maketitle
\centerline{\today}
\vspace{.2cm}

\noindent
{\bf Abstract}
The proximal point method for a special class of nonconvex multiobjective functions is studied in this paper. We show that the method is well defined and that the accumulation points of any generated sequence, if any, are Pareto--Clarke critical points. Moreover, under additional assumptions, we show the full convergence of the generated sequence.

\noindent
{\bf Keywords.} Multiobjective \,$\cdot$\,Pareto-Clarke optimality\,$\cdot$\,Nonconvex optimization.

\noindent{\bf AMS Classification.} 90C30\,$\cdot$\,90C29\,$\cdot$\,90C26.


\section{Introduction}
Multiobjective optimization is the process of simultaneously optimizing two or more real-valued objective functions. Usually, no single point will minimize all the given objective functions at once (i.e., there is no ideal minimizer), and so the concept of optimality 
has to be replaced by the concept of {\it Pareto optimality} or as we will see, {\it Pareto--Clarke critical}; see~\cite{Custodio2011}. These types of problems have applications 
in the economy, industry, agriculture, and other fields; see~\cite{Galtomas1997}. \cite{Bonnel2005} considered extensions of the proximal point method to the multiobjective setting, see also, 
\cite{ApolinarioPapaQuirozOliveira2016,BaoMordukhovich2010,BaoMordukhovich2009,Bentocruzneto2014,Ceng2010,Ceng2007,Chuong2011,Villacorta2011} and references therein. 
 
Our goal is to study the proximal point method introduced in \cite{Bonnel2005} for the multiobjective problems, where each component function is lower-$C^1$, a special class of 
nonconvex functions. Over the last four decades, several authors have proposed generalized proximal point methods for certain nonconvex minimization problems. As far as we know, the first  
generalization was performed in~\cite{Fukushima1981}, see also~\cite{Kaplan1998} for a review. 
Our approach extends to the multiobjective context the results of ~\cite{Kaplan1998}. More precisely, we show that the method is 
well defined and that the accumulation points of any generated sequence, if any, are Pareto--Clarke critical for the multiobjective function. Moreover, under some 
 additional assumptions, we show the full convergence of the generated sequence.

The organization of the paper is as follows.
In Section \ref{sec1}, some notation and basic results used throughout the paper are presented. In Section \ref{sec3}, the main results are stated and   proved. Some final remarks are made in Section~\ref{sec:conclusion}.
 \section{Preliminaries}\label{sec1}
In this section, we present some basic results and definitions. 

We denote $I:=\{1,\dots,m\}$, $\mathbb{R}^{m}_{+}:=\left\{x\in\mathbb{R}^m~:~x_j\geq0, j\in I\right\}$, and $\mathbb{R}^{m}_{++}:=\left\{x\in\mathbb{R}^m~:~x_j>0, j\in I\right\}$. 
For $y, z\in\mathbb{R}^m$, $z\succeq y$ (or $y\preceq z$ ) means that $z-y\in\mathbb{R}^{m}_{+}$ and $z\succ y$ (or $y\prec z$ ) means that $z-y\in\mathbb{R}^{m}_{++}$. 
We consider the {\it unconstrained multiobjective problem}:
$\min_{x\in\mathbb{R}^n}F(x)$, where $F:\mathbb{R}^n\rightarrow\mathbb{R}^m$, with $F(x)=(f_1(x),\dots,f_m(x))$.
Given a nonempty set $C\subset \mathbb{R}^{n}$, a point $x^*\in C$ is said to be a weak Pareto solution of the problem 
$\min_w \{F(x)~:~x\in C\}$ if, and only if, there is no $x\in C$ with $F(x)\prec F(x^*)$. We denote as $\mbox{argmin}_w \{F(x)~:~x\in C\}$ the weak Pareto solutions set. In particular, when $C=\mathbb{R}^{n}$, we denote this set as $U^{*}$.
Assume that $C$ is convex. $F$ is called {$\nu$-\it strongly convex} (or simply strongly convex) on $C$, $\nu\in\mathbb{R}^m_{++}$, if, and only if, for every $x, y\in C$, 
$$
F\left((1-t)x+ty\right)\preceq(1-t)F(x)+tF(y)- \nu t(1-t)\|x-y\|^2 ,\quad t\in[0,1].
$$
$F$ is said to be convex when $\nu=0$ in the above inequality.
Note that {\it F} is convex (resp. strongly convex) if, and only if, {\it F} is component-wise convex (resp. strongly convex). Moreover, this definition generalizes the definition of a convex function in the scalar case.
The proof of the next proposition can be found in~\cite[p. 95]{DinhLuc1989}.
\begin{proposition} \label{p:iam}
If $C$ is a convex set and $F$ is a convex function, then
$$
\bigcup_{z\in\mathbb{R}^{m}_{+}\backslash\{0\}} \emph{argmin}_{x\in C}\; \langle F(x),z\rangle= \emph{argmin}_w \{F(x)~:~x\in C\}. 
$$
\end{proposition}
If $m=1$, $f$ is $L$-strongly convex on $\Omega\subset\mathbb{R}^n$ with constant $L$ if, and only if,
\begin{equation}\label{eqq12}
\langle u-v,x-y\rangle\geq L \|u-v\|^2,\quad  u\in\partial f(x),\quad  v\in\partial f(y), 
\end{equation}
whenever $x, y \in \Omega$, where $\partial f$ denotes the subdifferential.
\begin{remark}\label{remarksubconvex}
 Let $f_1, f_2:\mathbb{R}^n\to \mathbb{R}$ be convex on $\Omega$. Thus, $\partial f_{1}(x)$ and $\partial f_{2}(x)$ are nonempty, convex, and compact for $x\in \Omega$. Moreover, if $\lambda_1, \lambda_2\geq 0$ then
$
\partial (\lambda_1f_1+\lambda_2f_2)(x)=\lambda_1\partial f_1(x)+\lambda_2 \partial f_2(x),
$ for $x\in \Omega$; see~\cite[Theorem 23.8]{Rockafellar1970}.
\end{remark}
Let $C\subset\mathbb{R}^n$ be nonempty, closed, and convex. The normal cone is defined by 
\begin{equation}  \label{remark1}
 N_C(x):=\{v\in\mathbb{R}^n~:~\langle v,y-x\rangle\leq0,~ y\in C\}.
\end{equation}
\begin{remark}  \label{rm.2015}
If $g:\mathbb{R}^{n}\to\mathbb{R}$ is convex, then the first-order optimality condition for $\min_{x\in C}g(x)$ is $0\in \partial g(x)+N_{C}(x)$. 
If $g$ is the maximum of a finite collection of continuously differentiable functions, then this constraint qualification is satisfied.
\end{remark}

 Let $f:\mathbb{R}^n\rightarrow\mathbb{R}$ be locally Lipschitz at $x\in \mathbb{R}^n$ and $d\in\mathbb{R}^n$. The Clarke directional derivative \cite[p. 25]{Clarke1990} of $f$ at $x$ in the direction $d$ is defined as 
$$
f^{\circ}(x,d):=\displaystyle\limsup_{t\downarrow0~y \rightarrow x}\frac{f(y+td)-f(y)}{t}, 
$$
and the Clarke subdifferential of $f$ at $x$, denoted by $\partial^{\circ}f(x)$, is defined as
$$
\partial^{\circ}f(x):=\left\{w\in\mathbb{R}^n~:~\langle w,d\rangle\leq f^{\circ}(x,d),~\forall~d\in\mathbb{R}^n\right\}.
$$
 
The previous definition can be found in \cite[p. 27]{Clarke1990}. If $f$ is convex, $f^{\circ}(x,d)=f'(x,d)$, where $f'(x,d)$ is the usual directional derivative. Moreover, $\partial^{\circ}f(x)=\partial f(x)$ for all $x\in\mathbb{R}^n$; see~\cite[Proposition 2.2.7]{Clarke1990}.
The next lemmas can be found in ~\cite[p. 39]{Clarke1990} 
\begin{lemma}\label{lema234556} Let $\Omega\subset\mathbb{R}^n$ be open and convex. If $f:\mathbb{R}^n\rightarrow\mathbb{R}$ is locally Lipschitz on $\Omega$ and $g:\mathbb{R}^n\rightarrow\mathbb{R}$ is convex on $\Omega$, then $(f+g)^{\circ}(x,d)=f^{\circ}(x,d)+g'(x,d)$ 
for each $ x\in\Omega$ and $d\in\mathbb{R}^n$. Consequently, if $g:\mathbb{R}^n\rightarrow\mathbb{R}$ is continuously differentiable on $\Omega$,
$ \partial^{\circ}(f+g)(x)=\partial^{\circ}f(x)+{\rm{grad}}\,g(x)$ for each $ x\in\Omega.$
\end{lemma}
\begin{lemma}\label{lemma2233} 
Let $\Omega\subset\mathbb{R}^n$ be open and convex. Let $f_i:\mathbb{R}^n\rightarrow\mathbb{R}$ be a continuously 
differentiable function on $\Omega$, $i\in I$. Define $f(x):=\max_{i\in I}f_i(x)$,
and $I(x):=\{i\in I:f_i(x)=f(x)\}$.
Then, \emph{\textbf{(a)}} $f$ is locally Lipschitz on $\Omega$ and 
$\mbox{conv}\{\emph{grad}f_i(x):i\in I(x)\}\subset\partial^{\circ}f(x)$, $x\in\Omega$;
\emph{\textbf{(b)}} if $f_i:\mathbb{R}^n\rightarrow\mathbb{R}$ is differentiable and convex on $\Omega$, $i\in I$, then
$\partial f(x)=\mbox{conv}\{\emph{grad}f_i(x):i\in I(x)\}$. In particular, $x$ minimizes $f$ on $\Omega$ if, and only if, there exists $\alpha_i\geq0$, $i\in I(x)$, such that
$0=\sum_{i\in I(x)}\alpha_i\,{\rm grad}\,f_i(x)$ and $\sum_{i\in I(x)}\alpha_i=1$;
\emph{\textbf{(c)}} if $f_i:\mathbb{R}^n\rightarrow \mathbb{R}$ is $L_i$-strongly convex, for $i\in I$, $f$ is $\min_{i\in I}L_i$ strongly convex. 
\end{lemma}
\begin{proof} 
The proofs of items (a) and (b) can be found in \cite[Proposition~{4.5.1}]{Bertsekasbook2003} and \cite[p. 49]{Maquela1992}, respectively. The proof of item (c) follows from the definition of a strongly convex function.
\end{proof}
\begin{definition} Let $F=(f_1,\dots,f_m)^{T}:\mathbb{R}^n\rightarrow\mathbb{R}^m$ be locally Lipschitz on $\mathbb{R}^n.$ We say that $x^*\in\mathbb{R}^n$ is a Pareto--Clarke critical point of $F$ if, for all directions $d\in\mathbb{R}^n$, there exists $i_0=i_0(d)\in\{1,\dots,m\}$, such that $f^{\circ}_{i_0}(x^*,d)\geq0.$
\end{definition}
\begin{remark} 
The previous definition can be found in \cite{Custodio2011}. When $m=1$, the last definition
 becomes the classic definition of the critical point for the nonsmooth convex function.
The last definition generalizes, for nonsmooth multiobjective optimization,
the condition
$
\mbox{Im}\left(JF(x^*)\right)\cap\left(-\mathbb{R}^{m}_{++}\right)=\emptyset,
$
 which characterizes a Pareto
critical point when $F$ is continuously differentiable. 
\end{remark}

\section{Proximal Algorithm for Multiobjective Optimization}\label{sec3}
In this section, we present a proximal point method to minimize a nonconvex function $F$, where its component is given by the maximum of continuously differentiable functions. Our goal is to prove the following theorem:
\begin{theorem}\label{theorem1} Let $\Omega\subset\mathbb{R}^n$ be open and convex and $I_j:=\{1,\ldots,\ell_j\}$, with $ \ell_j \in \mathbb{Z}_+$. Let $F(x):=(f_1(x),\ldots,f_m(x))$, where
$
f_j(x):=\max_{i\in I_j}f_{ij}(x)$, $j\in{I},
$ 
and $f_{ij}:\mathbb{R}^n\rightarrow\mathbb{R}$ is a continuously differentiable function on $\Omega$ and continuous on $\bar {\Omega}$, for all $i\in I_j$. Assume that for all $j\in{I}$, $-\infty<\inf_{x\in\mathbb{R}^n} f_j(x)$, ${\rm {grad}}f_{ij}$ is Lipschitz on $\Omega$ with constant $L_{ij}$ for each $i\in I_j$ and
$S_F(F(\bar y)):=\left\{x\in\mathbb{R}^n~:~ F(x)\preceq F(\bar{y})\right\}\subset\Omega$, for some $\bar y\in \mathbb{R}^n$.
Let $\bar{\lambda}>0$ and $\bar{\mu}>0$, such that $ \bar{\mu}<1$. Take  $\{e^k\}\subset \mathbb{R}^m_{++}$ and $\{\lambda_k\}\subset \mathbb{R}_{++}$ satisfying 
\begin{equation}\label{lambda}
\|e^k\|=1, \quad \bar{\mu}<e^k_j, \quad \frac{1}{\bar{\mu}}\max_{i\in I_j}L_{ij}<\lambda_k \leq\bar{\lambda},   \quad j\in {I}, \quad k=0, 1,\ldots.
\end{equation}
Let $\hat{x}\in S_F(F(\bar{y}))$. If $\Omega_k:=\{x\in\mathbb{R}^n ~:~ F(x)\preceq F(x^k)\}$, then 
\begin{equation}\label{method}
x^{k+1}\in\emph{argmin}_{w}\left\{F(x)+\frac{\lambda_k}{2}\|x-x^k\|^2e^k  ~:~ x\in\Omega_k\right\},\quad k=0,1,\ldots,
\end{equation}
starting with $x^0=\hat{x}$ is well defined, the generated sequence $\{x^k\}$ rests in $S_F(F(\bar{y}))$ and any accumulation point of $\{x^k\}$ is a Pareto--Clarke critical point of $F$, as long as $\Omega_k$ is convex, for each $k$. 
\end{theorem}
In order to prove the above theorem we need some preliminaries. Hereafter, we assume that all the assumptions
 of Theorem~\ref{p:iam} hold. We start proving the well-definedness of the sequence in \eqref{method}.
\begin{proposition} \label{pr:wdef}
 The proximal point method \eqref{method} applied to $F$ with starting point $x^0=\hat{x}$ is well defined.
\end{proposition}
\begin{proof} 
The proof will be made by induction on $k$. Let $\{x^k\}$ be as in \eqref{method}. By assumption, $\hat{x}\in S_F(F(\bar{y}))$. Thus, we assume that $x^k\in S_F(F(\bar{y}))$ for some $k$. Take $z\in\mathbb{R}_{+}^{m}\backslash\{0\}$ and 
define $\varphi_k(x):=\langle F(x),z\rangle+(\lambda_k/2)\langle e^k, z\rangle\|x-x^k\|^2$.
As $-\infty<\inf_{x\in\mathbb{R}^n} f_j(x)$ for all $j\in I$, the function $\langle F(\cdot),z\rangle$ is bounded below and, taking into account that $\langle e^k, z\rangle>0$, it follows that $\varphi_k$ is coercive. Then, as $\Omega_k$ is closed, there exists $\tilde{x}\in\Omega_k$, such that
$\tilde{x}=\mbox{argmin}_{x\in \Omega_k}\varphi_k(x)$.
Therefore, from Proposition~\ref{p:iam} we can take $x^{k+1}:=\tilde{x}$ and the induction is done, proving the proposition. \qed
\end{proof}
\begin{lemma} \label{lemma1}For all $\tilde{x}\in\mathbb{R}^n$,   $v:=(v_1, \ldots, v_m) \in \mathbb{R}^m_{++}$, $ j\in{I}$ and $\lambda$ 
satisfying $\sup_{i\in I_j}L_{ij}<\lambda v_j$, the functions $f_{ij}+\lambda v_j\|\cdot-\tilde{x}\|^2/2$, $f_{j}+\lambda v_j \|\cdot-\tilde{x}\|^2/2$ and 
$F+(\lambda/2)\|\cdot-\tilde{x}\|^2v$ are strongly convex on $\Omega$.  Moreover, $ \left\langle F(\cdot),z \right\rangle+\lambda\left\langle v, z\right\rangle \left\| \cdot -\tilde{x}\right\|^2/2$
is strongly convex on $\Omega$ for each $ z \in  \mathbb{R}^m_+\backslash\{0\}$.
\end{lemma}
\begin{proof} Take $ j\in{I}$, $i\in I_j$, $\tilde{x}\in\mathbb{R}^n$, $v_j \in \mathbb{R}_{++}$ and define $h_{ij}=f_{ij}+\lambda v_j \|\cdot-\tilde{x}\|^2/2$. Since ${\rm grad}\,h_{ij}(x)={\rm grad}\,f_{ij}(x)+\lambda v_j (x-\tilde{x})$,   
we have $\langle {\rm grad}\,h_{ij}(x)-{\rm grad}\,h_{ij}(y),x-y\rangle=\langle {\rm grad}\,f_{ij}(x)-{\rm grad}\,f_{ij}(y),x-y\rangle+\lambda v_j \|x-y\|^2\ $. Using the Cauchy inequality, last equality becomes 
\begin{eqnarray*}
\langle {\rm grad}\,h_{ij}(x)-{\rm grad}\,h_{ij}(y),x-y\rangle &\geq&-\|{\rm grad}\,f_{ij}(x)-{\rm grad}\,f_{ij}(y)\|\|x-y\|\\
&+&\lambda v_j \|x-y\|^2. 
\end{eqnarray*}
As ${\rm {grad}}f_{ij}$ is Lipschitz on $\Omega$ with constant $L_{ij}$,  
$\langle {\rm grad}\,h_{ij}(x)-{\rm grad}\,h_{ij}(y),x-y\rangle\geq (\lambda v_j  -L_{ij})\|x-y\|^2$ holds.
Hence, the last inequality along with the assumption $\lambda v_j >\sup_{i\in I_j}L_{ij}$ implies that ${\rm grad}\,h_{ij}$ is strongly monotone. Therefore, \eqref{eqq12} implies that $h_{ij}$ is strongly convex, proving the first part of the lemma. The second and third parts of the lemma follow from the first one.
\qed
\end{proof}
Hereafter, $\{x^k\}$ is generated by \eqref{method}. Note that  Proposition~\ref{p:iam} implies that there exists a sequence $\{ z^k\}\subset   \mathbb{R}^m_+\backslash\{0\}$, such that   
\begin{equation}\label{eq:argmin}
x^{k+1}=\mbox{argmin}_{x\in\Omega_k}\psi_k(x),
\end{equation}
where the function $\psi_k:\mathbb{R}^n\rightarrow\mathbb{R}$ is defined by 
\begin{equation}\label{eq:psi}
\psi_k(x):=\left\langle F(x),z^k \right\rangle+\frac{\lambda_k}{2}\left\langle e^k, z^k\right\rangle \left\|x-x^k\right\|^2.
\end{equation}
The solution of the problem in \eqref{eq:argmin} is not altered through the multiplication of $z^k$ by positive scalars. Thus, we can suppose $\|z^k\|=1$ for $k=0,1, \ldots.$

\begin{ProofMainTheo} 
The well-definedness of \eqref{method} follows from  Proposition~\ref{pr:wdef}. As $x_0=\hat{x}\in S_F(F(\bar{y}))\subset\Omega$, \eqref{method}
 implies $\{x^k\}\subset  S_F(F(\bar{y}))$. Let $\bar{x}$ be an accumulation point of $\{x^k\}$. Assume that $\Omega_k$ is convex and, by contradiction, that $\bar{x}$ is not Pareto--Clarke critical of $F$. Then, 
there exists $d\in\mathbb{R}^n$, such that
\begin{equation}\label{ineqqqqq}
f_{i}^{\circ}(\bar{x},d)<0,\quad i\in I.
\end{equation} 
Thus, $d$ is a descent direction for $F$ in $\bar{x}$ and there exists $\delta>0$, such that $F(\bar{x}+td)\prec F(\bar{x})$ for all $t\in(0,\delta]$. Hence, $\bar{x}+td\in\Omega_k$, for $k=0, 1, \dots$.  

Let $\{z_k\}$ be a sequence satisfying \eqref{eq:argmin}. Hence, we can combine Lemma \ref{lemma1} and Remark~\ref{rm.2015} to obtain 
$$
0\in\partial\left(\left\langle F(\cdot),z^k\right\rangle+\frac{\lambda_k}{2}\left\langle e^k, z^k\right\rangle\left\|\cdot-x^k\right\|^2\right)(x^{k+1})+N_{\Omega_k}(x^{k+1}), \quad k=0,1,\dots.
$$
Letting $z^k=(z^k_1,\dots,z^k_m)$ and $e^k=(e^k_1,\dots,e^k_m)$, Remark \ref{remarksubconvex} gives us,
$$
0\in\sum_{j=1}^{m}z^k_j\partial\left(f_j+\frac{\lambda_k}{2} e^k_j\left\|\cdot-x^k\right\|^2\right)(x^{k+1})+N_{\Omega_k}(x^{k+1}), \quad k=0,1,\dots.
$$
The last inclusion implies that there exists $v^{k+1}\in N_{\Omega_k}(x^{k+1})$, such that
$$
0\in\sum_{j=1}^{m}z^k_j\partial\left(f_j+\frac{\lambda_k}{2}e^k_j\left\|\cdot-x^k\right\|^2\right)(x^{k+1})+v^{k+1}, \quad k=0,1,\dots.
$$
Since $\max_{i\in I_j}L_{ij}<\lambda_ke^k_j$, Lemma \ref{lemma1} implies that $f_{ij}+\lambda_k e^k_j\|\cdot-x^k\|^2/2$ and $f_j+\lambda_k e^k_j\|\cdot-x^k\|^2/2$ are strongly convex for all $j\in {I}$, $k=0, 1,\dots$. 
Applying Lemma~\ref{lemma2233}(b), for $I=I_j$ and for the functions $f_{ij}+\lambda_k e^k_j \|\cdot-x^k\|^2/2$ and $f_j+\lambda_k e^k_j\|\cdot-x^k\|^2/2$, for each $j\in{I}$, we obtain
\begin{eqnarray*}
0&=&\sum_{j=1}^{m}z^k_j\left(\sum_{i\in I_j(x^{k+1})}\alpha_{ij}^{k+1}{\rm{grad}}\left(f_{ij}+\frac{\lambda_k e^k_j}{2}\left\|\cdot-x^k\right\|^2\right)(x^{k+1})\right)+v^{k+1},\nonumber\\
& &\sum_{i\in I_j(x^{k+1})}\alpha_{ij}^{k+1}=1,
\end{eqnarray*}
which holds for all $k=0,1,\dots$, with $\alpha_{ij}^{k+1}\geq0$, $i\in I_j(x^{k+1})$. This tells us that
\begin{eqnarray}\label{eqmain}
0&=&\sum_{j=1}^{m}z^k_j\left(\sum_{i\in I_j(x^{k+1})}\alpha_{ij}^{k+1}\left({\rm{grad}}f_{ij}(x^{k+1})+\lambda_k e^k_j (x^{k+1}-x^k)\right)\right)+v^{k+1},\nonumber\\
& &\sum_{i\in I_j(x^{k+1})}\alpha_{ij}^{k+1}=1,
\end{eqnarray}
for all $k=0,1,\dots$. For all $j\in{I}$, let $\{\alpha_{ij}^{k+1}\}\subset\mathbb{R}^m$ be the sequence defined by
$\alpha_{j}^{k+1}=(\alpha_{1j}^{k+1},\alpha_{2j}^{k+1},\dots,\alpha_{mj}^{k+1})$, $\alpha_{ij}^{k+1}=0$, $i\in I_j\backslash I_j(x^{k+1})$,
for all $k=0,1,\dots$. Since $\sum_{i\in I_j(x^{k+1})}\alpha_{ij}^{k+1}=1$, $\|\alpha_j^{k+1}\|_1=1$ for all $k$, where $\|\cdot\|_1$ is the sum norm in 
$\mathbb{R}^n$. Thus, $\{\alpha_{j}^{k+1}\}$ is bounded. 
As $\{x^k\}\subset  S_F(F(\bar{y}))$ and $F$ is continuous on $\Omega$, we have $\bar{x}\in S_F(F(\bar{y}))$. Since $I_j$ is finite we can assume without loss of generality that
$I_j(x^{k_1+1})=I_j(x^{k_2+1})=\cdots=:\tilde{I}_J$,
and \eqref{eqmain} becomes
\begin{eqnarray}\label{eqmain2}
0&=&\sum_{j=1}^{m}z^{k_s}_j\left(\sum_{i\in\tilde{I}_J}\alpha_{ij}^{k_s+1}{\rm{grad}}f_{ij}(x^{k_s+1})+\lambda_{k_s} e^{k_s}_j(x^{k_s+1}-x^{k_s})\right)+v^{k_s+1},\nonumber\\
& &\sum_{i\in\tilde{I}_J}\alpha_{ij}^{k_s+1}=1,\quad s=0,1,\dots.
\end{eqnarray}
From the continuity of $F$ we obtain that $\Omega_k$ is closed. Considering that $x^{k_s}\in\Omega_{k_s}$,  $\Omega_{k_s}$ is convex and $\Omega_{k_s+1}\subset\Omega_{k_s}$, for $s=0, 1, \ldots$, we obtain that
\begin{equation} \label{eq:obar}
\tilde{\Omega}:=\cap_{s=0}^{+\infty}\Omega_{k_s},
\end{equation}
is nonempty, closed, and convex. As $v^{k_s+1}\in N_{\Omega_{k_s}}(x^{k_s+1})$ and $\tilde{\Omega}\subset\Omega_{k_s}$,  \eqref{remark1} implies
\begin{equation}\label{ineqeee}
\langle v^{k_s+1},~x-x^{k_s+1}\rangle\leq0,\qquad x\in\tilde{\Omega}, \qquad s=0,1,\dots.
\end{equation}
On the other hand, let $\{z^{k_s+1}\}$,  $\{x^{k_s+1}\}$, $\{e^{k_s+1}_j\}$,  $\{\lambda_{k_s+1}\}$, and $\{\alpha_j^{k_s+1}\}$ be the subsequences of $\{z^{k+1}\}$,  $\{x^{k+1}\}$,  $\{e^{k+1}_j\}$,  $\{\lambda_{k+1}\}$, and $\{\alpha_j^{k+1}\}$, respectively, such that 
$\lim_{s\to+\infty}(z^{k_s+1},x^{k_s+1},e_j^{k_s+1},\lambda_{k_s+1},\alpha_j^{k_s+1})=(\overline{z},\overline{x},\overline{e}_j,\hat{\lambda},\bar{\alpha}_j)$.
This fact along with \eqref{eqmain2}, implies that $\lim_{s\to+\infty}v^{k_s+1}=\bar{v}$. From \eqref{ineqeee}, $\bar{v}\in N_{\tilde{\Omega}}(\bar{x})$. Hence, in view of \eqref{eqmain2}, we have
$0=\sum_{j=1}^{m}{\bar z}_j\sum_{i\in\tilde{I}_J}\bar{\alpha}_{ij}{\rm{grad}}f_{ij}(\bar{x})+\bar{v}$ and $ \sum_{i\in\tilde{I}_J}\bar{\alpha}_{ij}=1$.
Let $x\in\tilde{\Omega}$. Taking $u_j=\sum_{i\in\tilde{I}_J}\bar{\alpha}_{ij}{\rm{grad}}f_{ij}(\bar{x})$, we have
\begin{equation} \label{eq:feee}
0=\sum_{j=1}^{m}{\bar z}_j\langle u_j,x-\bar{x}\rangle+\langle \bar{v},x-\bar{x}\rangle.
\end{equation}
As $\bar{x}+td\in\Omega_k$, for all $k=0, 1, \ldots$, the definition of $\tilde{\Omega}$ in \eqref{eq:obar} implies that $\bar{x}+td\in\tilde{\Omega}$, $t\in(0,\delta]$. Since $u_j=\sum_{i\in\tilde{I}_J}\bar{\alpha}_{ij}{\rm{grad}}f_{ij}(\bar{x})$ and 
$\sum_{i\in\tilde{I}_J}\bar{\alpha}_{ij}=1$, Lemma~\ref{lemma2233} (a) and (b) implies that $u_j\in \partial^{\circ}f_j(\bar{x})$. Hence, using that  $\bar{v}\in N_{\tilde{\Omega}}(\bar{x})$ and definition of  $f_j^{\circ}(\bar{x},d)$, equality \eqref{eq:feee} with $x=\bar{x}+td$ yields
$
0\leq\sum_{j=1}^{m}{\bar z}_j\langle u_j,d\rangle\leq \sum_{j=1}^{m}{\bar z}_j f_j^{\circ}(\bar{x},d).
$
Thus, there exists $j\in I$ such that $f_j^{\circ}(\bar{x},d)\geq0$, which contradicts \eqref{ineqqqqq}. Therefore, $\bar{x}$ is Pareto--Clarke critical point of $F$. 
 \end{ProofMainTheo} 

Now let us introduce some conditions that will guarantee that $\{x^k\}$ converges to a point $x^*\in U^*$. Suppose that 
\begin{description}
\item[(H1)] $U=\{y\in\mathbb{R}^n~:~ F(y) \preceq F(x^k), ~ k=0,1, \ldots\}\neq \varnothing$;
\item[(H2)] there exists $c\in \mathbb{R}$ such that the following conditions hold:
 \begin{description}
 \item[(a)] $S_F(ce):=\left\{x\in\mathbb{R}^n~:~ F(x)\preceq ce\right\} \neq \varnothing$ and $S_F(ce) \subsetneq S_{F}(F(\bar{y}))$;
 \item[(b)] $S_F(ce)$ is convex and $F$ is convex on $S_{F}(ce)$, where $e:=(1, \ldots, 1 )\in \mathbb{R}^m$;
 \end{description}
\item[(H3)] there exists $\delta>0$ such that for all $z\in\mathbb{R}^m_+\backslash \{0\}$, $x\in S_{F}(F(\bar{y}))\,\backslash\, S_{F}(ce)$ and $w_z(x)\in\partial^{\circ}\left(\langle F(\cdot),z\rangle\right)(x)+N_{\Omega_k}(x)$, it holds that $\|w_z(x)\|>\delta>0$.
\end{description}
In general, the set $U$ defined in assumption (H1) may be an empty set. To guarantee that $U$ is nonempty, an additional assumption on the sequence $\{x^{k}\}$ is needed. In the next remark we give such a condition.
\begin{remark} \label{eq:noset}
If the sequence $\{x^{k}\}$ has an accumulation point, then $U\neq \varnothing$, i.e., assumption (H1) holds. Indeed, let $\bar{x}$ be an accumulation point of the sequence 
$\{x^{k}\}$. Then, there exists a subsequence $\{x^{k_j}\}$ of $\{x^{k}\}$ which converges to $\bar{x}$. Since $F$ is continuous,  
$\{F(x^{k})\}$ has $F(\bar{x})$ as an accumulation point. Using the definition of $\{x^{k}\}$ in \eqref{method}, we conclude that  
$\{F(x^{k})\}$ is a decreasing sequence. Hence, the usual arguments easily show that the whole sequence $\{F(x^{k})\}$ converges to $F(\bar{x})$  
and $\bar{x}\in U$, i.e., $U\neq \varnothing$.  
\end{remark}
Next, we present a function satisfying the assumptions of Theorem~\ref{theorem1}, as well as {\rm{(H1)}}, {\rm{(H2)}} and {\rm{(H3)}}, when $m=2$, $j=2$ and ${I}=I_{j}= 2$.  
\begin{example} \label{example}
  Take $0<\epsilon<0.4$, $\Omega=(\epsilon, +\infty)$, ${I}:=\{1,2\}$, and $\bar y=2.718 \ldots $ with $\ln \bar y=1$. Let $F:\mathbb{R}\to\mathbb{R}^2$ be defined by $F(x)= (0,0)$ for $x\in \mathbb{R} \backslash \mathbb{R}_{++}$ and $F(x):=(f_1(x),f_2(x))$, where 
	$f_j(x):=\max_{i\in {I} }f_{ij}(x)$ for $j\in{I}$ and 
$f_{11}(x)=\ln x+1/x$, $f_{21}(x)=\ln x-1/x$, $f_{12}(x)=2\sqrt{x}+1/x$, $f_{22}(x)=2\sqrt{x}-1/x$, 
for all $x\in \mathbb{R}_{++}$. Note that $f_{1j}$, $f_{2j}$ are continuously differentiable on $\Omega$ and continuous on $\bar {\Omega}$, for all  $j\in{I}$. Since $f''_{1j}$, $f''_{2j}$ are bounded on $\Omega$, we conclude that 
$f'_{1j}$, $f'_{2j}$ are Lipschitz on $\Omega$, for $j\in{I}$. Since $\max\{a,b\}=(a+b)/2+|a-b|/2$, for all $a, b\in  \mathbb{R}$, we conclude $f_1(x)=\ln x+1/x$ and $f_2(x)=2\sqrt{x}+1/x$. Therefore, 
$
F(x)=\left(\ln x+1/x, ~ 2\sqrt{x}+1/x\right)$, $x\in \mathbb{R}_{++}$. 
It is easy to see that $F$ is nonconvex and $\{x\in\Omega~:~ F(x)\preceq (\zeta, \zeta)\} $ is convex and nonempty, for all $\zeta\geq 1$. Consider the following multiobjective optimization problem  
$\min_{w} \{ F(x) ~:~ x\in \Omega\}$,
which has $x^*=1$ as the unique solution. In fact, $F(1) \prec  F(x)$, for all $x\in \mathbb{R}_{++}$. Hence, we obtain that $-\infty<\inf_{x\in\mathbb{R}}f_j(x)$, for $j\in{I}$. Since $0<\epsilon<\bar y$, we conclude that   
$S_F(F(\bar y))\subset\Omega$ and $S_F(F(\bar y))\neq \varnothing$.  
Therefore, taking into account that $\Omega_k$ is convex, $F$ satisfies all the assumptions of Theorem~\ref{theorem1}.    
We are going to prove that $F$ also satisfies {\rm{(H1)}}, {\rm{(H2)}}, and {\rm{(H3)}}. Since $F(1) \prec  F(x)$, for all $x\in \mathbb{R}_{++}$, we conclude that  
 $F$ satisfies {\rm{(H1)}}. Let $c=f_2(2)$ and note that  
$
 (0.6,  ~2] \subset  S_F(ce) \subsetneq  [0.5, ~2.7] \subset S_F(F(\bar y)).
 $
But this tells us, in particular, that $F$ satisfies {\rm{(H2)}}. Finally, we are going to prove that $F$ satisfies {\rm{(H3)}}. First, note that 
$
S_F(F(\bar y))\backslash S_F(ce)\subset [0.47, ~0.57)\cup(2, ~2.72].
$
For each point $z=(z_1, z_2)\in\mathbb{R}^2_+\backslash \{0\}$ with $\|z\|_{1}:=z_1+z_2=1$, take $x\in S_F(F(\bar y))\backslash S_F(ce)$ and 
$w_z(x)\in\partial^{\circ}\left(\langle F(\cdot),z\rangle\right)(x)+N_{\Omega_k}(x)$. Hence, there exists $v\in N_{\Omega_k}(x)$, such that 
\begin{equation} \label{eq:exeq1}
w_z(x)=z_1\varphi_1(x)+z_2\varphi_2(x)+v,
\end{equation}
where $\varphi_1(x):=(1/x-1/x^2)$ and $\varphi_2(x):=(1/\sqrt{x}-1/x^2)$.
First, we assume that $x\in[0.47, ~0.57)$. In this case,  $ N_{\Omega_k}(x)\subset\mathbb{R}_{-}$ and, using the above equality, we obtain
$w_z\leq z_1\varphi_1(x)+z_2\varphi_2(x)$.
 Since $(x-1)/x^2<-0,4/(0,47)^2$ and $(x^2-\sqrt{x})/x^{3/2}<-0,2/(0,47)^{3/2}$, we have
$
w_z(x)\leq z_1\varphi_1(x)+z_2\varphi_2(x)<-0,4/(0,47)^2z_1-0,2/(0,47)^{3/2}z_2.
$
Then, for all $w_z(x)\in\partial^{\circ}\left(\langle F(\cdot),z\rangle\right)(x)+N_{\Omega_k}(x)$,  
\begin{equation} \label{eq:pex2}
|w_z(x)|>\frac{0,4}{(0,47)^2}z_1+\frac{0,27}{(0,47)^{3/2}}z_2>\|z\|_{1}\frac{0,27}{(0,47)^{3/2}}=\frac{0,27}{(0,47)^{3/2}}, 
\end{equation}
for $x\in[0.47, ~0.57)$. Assuming that $x\in(2, ~2.72]$, it follows that  $ N_{\Omega_k}(x)\subset\mathbb{R}_{+}$. Hence, it follows from \eqref{eq:exeq1} that 
$w_z(x)\geq z_1\varphi_1(x)+z_2\varphi_2(x)$.
From $(x-1)/x^2>1/(2,72)^2$ and $\left(1/\sqrt{x}-1/x^2\right)>2,3/(2,72)^{3/2}$ we obtain 
$
w_z(x)\geq z_1\varphi_1(x)+z_2\varphi_2(x)>1/(2,72)^2z_1+2,3/(2,72)^{3/2}z_2.
$
Thus, for all $w_z(x)\in\partial^{\circ}\left(\langle F(\cdot),z\rangle\right)(x)+N_{\Omega_k}(x)$,
$$
|w_z(x)|>\frac{1}{(2,72)^2}z_1+\frac{2,3}{(2,72)^{3/2}}z_2>\|z\|_{1}\frac{1}{(2,72)^2}=\frac{1}{(2,72)^2}, \quad x\in(2, ~2.72].
$$
Since $S_F(F(\bar y))\backslash S_F(ce)\subset [0.47, ~0.57)\cup(2, ~2.72] $, combining \eqref{eq:pex2} with the last inequality, we conclude that $F$ satisfies {\rm{(H3)}} with $\delta= 1/(2,72)^2$.
\end{example}
\begin{lemma}  \label{eq:c2}
 Assume that {\rm (H1)}, {\rm (a)} in {\rm (H2)}, and {\rm (H3)} hold and $\lambda_k$ satisfies \eqref{lambda}.  
Then, after a finite number of steps the proximal iterates lies in $S_F(ce)$, i.e., there exists $k_0$ such that $\{x^k\}\subset S_F(ce)$, for all $k\geq k_0$.
\end{lemma}
\begin{proof} 
Condition {\rm (a)} in {\rm (H2)} implies that $S_F(ce)\neq \varnothing$. Suppose, by contradiction, that $x^k\in S_F(F(\bar y))\backslash S_F(ce)$ for all $k$. Let $\{z_k\}$ be a sequence satisfying \eqref{eq:argmin}. 
Hence, we can combine Lemma~\ref{lemma1}, Remark~\ref{rm.2015} and Lemma~\ref{lema234556} to obtain 
$0\in\partial^{\circ}\left(\left\langle F(\cdot),z^k\right\rangle\right)(x^{k+1})+(\lambda_k/2)\left\langle e^k, z^k\right\rangle\left(x^{k+1}-x^k\right)+N_{\Omega_k}(x^{k+1})$, $k\geq 0$.
Then, 
$-(\lambda_k/2)\left\langle e^k, z^k\right\rangle\left(x^{k+1}-x^k\right)\in\partial^{\circ}\left(\left\langle F(\cdot),z^k\right\rangle\right)(x^{k+1})+N_{\Omega_k}(x^{k+1})$, for $k\geq 0$.
As $x^{k+1}\in S_F(F(\bar y))\backslash S_F(ce)$, (H3) along with the last inclusion gives us
\begin{equation}\label{eq9898}
\frac{\lambda_k}{2}\left\langle e^k, z^k\right\rangle\left\|x^{k+1}-x^k\right\|>\delta,\quad k\geq 0.
\end{equation}
Hence, $(\lambda_k/2)\left\langle e^k, z^k\right\rangle \left\|x^{k+1}-x^k\right\|^2\leq \|F(x^{k})- F(x^{k+1})\|$ holds, in view of \eqref{eq:argmin} and \eqref{eq:psi}, and $\| z^k\|=1$, 
 for all $k\geq0$. Thus, from {\rm (H1)}, $ \|F(x^{k})- F(x^{k+1})\|\to0$, contradicting~\eqref{eq9898}.
\end{proof}

\begin{lemma}  \label{eq:c1}
Assume that {\rm (H1)} and {\rm (H2)} hold and $\lambda_k$ satisfies  \eqref{lambda}. If $x^k\in S_F(ce)$ for some $k$, then $\{x^k\}$ converges to a point $x^*\in U^*.$
\end{lemma}
\begin{proof} By hypothesis, $x^{k}\in S_F(ce)$ for some $k$, i.e., there exists $k_0$ such that $F(x^{k_0})\preceq ce$. Hence, the definition of $\{x^k\}$ in \eqref{method} 
implies that $\{x^k\}\subset S_F(ce)$, for all $k\geq k_0$. Therefore, using \eqref{lambda}, {\rm (H1)}, and {\rm (H2)}, the result follows by applying  
\cite[ Theorem~{3.1}]{Bonnel2005} with $F_0=F$, $S=S_F(ce)$,   $X=\mathbb{R}^n$,    $C=\mathbb{R}^m_+$, and using {\rm (H1)} instead of (A). 
\end{proof}
\begin{theorem}  \label{th:cont}
Under the conditions $(H1)$, $(H2)$, and $(H3)$, the sequence $\{x^k\}$ generated by \eqref{method} converges to a point $x^*\in U^*$.
\end{theorem}
\begin{proof}
It follows by combination of Lemma \ref{eq:c2} with Lemma \ref{eq:c1}.
\end{proof}
\begin{remark}
As the function in Example 3.1 is not convex, the analysis in \cite{Bonnel2005} does not allow us to conclude that $\{x^k\}$ converges to a minimizer. 
However, as the function satisfies (H1) to (H3), Theorem~\ref{th:cont} guarantees that $\{x^k\}$ converges.
\end{remark}

\section{Conclusions} \label{sec:conclusion}
The main contribution of this paper is the extension of the convergence analysis of the proximal method \eqref{method}, which has been studied in \cite{Bonnel2005}, in order to increase the range of its applications; see Example~\ref{example}. 
The proximal point method is indeed a conceptual scheme that transforms a given problem into a sequence of better behaved subproblems, which has been proven to be an efficient tool in several instances through the methods that can be derived from it 
(e.g., Augmented Lagrangians, both classical or generalized). In this sense, the proximal point method is basic and we expect that the results of the present paper will become an additional step toward solving general multiobjective optimization problems.  
Finally, it is worth noting that in our analysis of the preference relation was induced by a cone with a nonempty interior. Various vector optimization problems can be formalized by using convex ordering cones with empty interiors, this   
restriction could open some new perspectives in the point of view of the numerical methods, see  \cite{BaoMordukhovich2010,GradPop2014,Mordukhovich2006}. We foresee new developments in this direction in the near future.




\end{document}